\documentclass[11pt,a4paper,twoside]{amsart}
\usepackage[T1]{fontenc}
\usepackage[utf8]{inputenc}
\usepackage[english]{babel}
\usepackage{amsmath}
\usepackage{amsthm}
\usepackage{amssymb}
\usepackage{amsfonts}
\usepackage{amsxtra}
\usepackage{enumerate}
\usepackage{verbatim}
\usepackage{color}
\usepackage[mathscr]{eucal}
\usepackage{graphicx}
\usepackage{xcolor}
\usepackage[backref=page]{hyperref}
\usepackage{tikz}
\usepackage{enumitem}
\usepackage{parskip}
\usepackage{dirtytalk}

\newcommand{\Z}{\mathbb Z}

\newcommand{\R}{\mathbb R}

\newtheorem{theorem}{Theorem}[section]

\newtheorem{lemma}[theorem]{Lemma}

\newtheorem{remark}[theorem]{Remark}

\theoremstyle{definition}
\newtheorem{exmp}[theorem]{Example}

\let\phi=\varphi

\newcommand{\rem}[1]{}

\DeclareFontFamily{U}{mathb}{\hyphenchar\font45}
\DeclareFontShape{U}{mathb}{m}{n}{
<-6> mathb5 <6-7> mathb6 <7-8> mathb7
<8-9> mathb8 <9-10> mathb9
<10-12> mathb10 <12-> mathb12
}{}
\DeclareSymbolFont{mathb}{U}{mathb}{m}{n}
\DeclareMathSymbol{\llcurly}{\mathrel}{mathb}{"CE}
\DeclareMathSymbol{\ggcurly}{\mathrel}{mathb}{"CF}

\title{A non-orderable overtwisted contact structure on the sphere}

\author{Jakob Hedicke}
\address{ Radboud Universiteit , Heyendaalseweg 135, 6525 AJ Nijmegen, The Netherlands} 
\email{jakob.hedicke@gmail.com}

\date{\today}
\begin{document}

\begin{abstract}
We show that the overtwisted contact structure on $S^3$ with Hopf invariant $-1$ is non-orderable, i.e., that it admits a contractible positive loop of contactopmorphisms.
This provides the first known example of a non-orderable overtwisted contact manifold.
We further show the existence of a contactomorphism without translated points.
\end{abstract} 

\maketitle

\section{Introduction}

In their seminal paper \cite{Eliashberg00} Eliashberg and Polterovich introduced the notion of orderability for contact manifolds.
They call a closed cooriented contact manifold $(M,\zeta)$ orderable, if the universal cover of the identity component of the group of contactomorphisms $\widetilde{\mathcal{G}}:=\widetilde{\mathrm{Cont}}_0(M,\zeta)$ admits a bi-invariant partial order which is defined using paths of contactomorphisms positively transverse to the contact distribution.
A long-standing open problem in contact topology is the question whether overtwisted contact structures are orderable or not \cite[Section 1.9]{Eliashberg00}.
Even though there has been a lot of evidence that overtwisted contact manifolds are not orderable \cite{Borman15, Casals16, Liu20}, there has not been a single example in the literature of an overtwisted contact manifold that is known to be either orderable or non-orderable.

In this note we provide the first example of an overtwisted contact manifold that is non-orderable.

Combining the methods developed in \cite{Hedicke24} with results about Legendrian knots in $S^3$ \cite{Dymara01, Eliashberg09, Etnyre13, Vogel18}, we construct the following example.
Recall that overtwisted contact structures on $S^3$ are classified by their Hopf invariant $H\in \Z$, see e.g. \cite[Section 5]{Geiges24} for a detailed review.
Throughout this paper, $\xi$ will denote the overtwisted contact structure with Hopf invariant $-1$ (with respect to the quaternionic frame).
This contact structure can be obtained from the standard tight contact structure by performing a Lutz-twist around a fibre of the Hopf fibration.

\begin{theorem}\label{thmmain}
The contact manifold $(S^3,\xi)$ is non-orderable, i.e., it admits a contractible positive loop of contactomorphims.
\end{theorem}

\begin{remark}
\begin{enumerate}
\item We expect that a combination of the results in \cite{Hedicke24} with methods from \cite{Borman15} can be used to prove non-orderability for more general classes of overtwisted contact manifolds.
This will be part of a future work \cite{Hedickefuture}.
\item Similar knot theoretic methods as the ones used to prove Theorem \ref{thmmain} have been developed for certain overtwisted Lens spaces \cite{Geiges15, Rima25} and might be used to obtain further non-orderable examples in these cases in the same way.
\end{enumerate}
\end{remark}

Recall that a \textit{translated point} of a contactomorphism $\varphi\colon (M,\zeta)\rightarrow (M,\zeta)$ with respect to a contact form $\alpha$ is a point $x\in M$ such $(\varphi^{\ast}\alpha)_x=\alpha_x$ and $\varphi(x)=\varphi_t^{\alpha}(x)$ for some $t\in\R$ and the Reeb flow $\varphi_t^{\alpha}$ of $\alpha$, see \cite{Sandon12}.
Recent work seems to suggest that there is a strong connection between (non-) orderability and the (non-) existence of translated points, see e.g. \cite{Albers15, Cant22, Cant25, Hedicke24}.
In view of Theorem \ref{thmmain}, the following result reinforces this impression.

\begin{theorem}\label{thm_translated}
There exists a contact form $\alpha$ for $(S^3,\xi)$ and a contactomorphism $\varphi$ that has no translated points with respect to $\alpha$.
\end{theorem}

\subsection*{Outline of the proof of Theorem \ref{thmmain}}
The key idea of the proof is to use \cite[Theorem 1.10]{Hedicke24} which states that a contact manifold $(M,\zeta)$ is non-orderable if it is supported by an open book decomposition such that the skeleton of the page $L$ with respect to some adapted contact form can be displaced  by a contactomorphism from its image under the Reeb flow.
A canonical abstract open book for $(S^3,\xi)$ is given by an annulus page with a negative Dehn twist as monodromy.
In this case the skeleton is a non-loose Legendrian unknot in $S^3$ and its image under the Reeb flow is a pre-Lagrangian torus.
This allows to apply results on the classification of non-loose Legendrian unknots in $S^3$ \cite{Dymara01, Eliashberg09, Etnyre13, Vogel18} in order to find a contactomorphism that displaces $L$.

\subsection*{Acknowledgements}
I thank Dylan Cant, Georgios Dimitroglou Rizell, Yakov Eliashberg, Lukas Nakamura, Stefan Nemirovski, Murat Sa\u{g}lam and Egor Shelukhin for interesting discussions on related topics.
The author is supported by a Radboud Excellence Fellowship.
Part of this research was conducted while the author was visiting the Simons Center for Geometry and Physics for the program \say{Contact geometry, general relativity and thermodynamics}.

\section{Background}
\subsection{Overtwisted contact manifolds}
By a \textit{contact manifold} $(M,\zeta)$ we mean a $(2n+1)$-dimensional manifold $M$ together with a hyperplane bundle $\zeta\subset TM$ that is maximally non-integrable in the sense that $\zeta=\ker\alpha$ for some differential $1$-form satisfying that $\alpha\wedge (d\alpha)^n$ is a volume form.
Here we will focus on the case $n=1$.

A special class of contact manifolds satisfying a certain $h$-principle are \textit{overtwisted} contact manifolds, which where introduced in \cite{Eliashberg89} in dimension $3$ and in \cite{Borman15} for higher dimensions.
Roughly, in the $3$-dimensional case a manifold is called overtwisted, if it contains an embedded disk with Legendrian boundary whose characteristic foliation, i.e., the line bundle induced by the intersection of the contact structure with the tangent space of the disk, has a unique singular point in the interior, see \cite[Fig. 4.10]{Geiges}.
A contact manifold without an overtwisted disk is called tight.
Throughout this paper we will restrict to the closed $3$ dimensional case, where the classification of overtwisted contact structures up to isotopy coincides with the classification of oriented plane fields \cite{Eliashberg89}.

\begin{exmp}
On $S^3$ there exists an integer family of overtisted contact structures labelled by their Hopf invariant, one for each isotopy class of oriented plane fields.
The unique tight contact structure on $S^3$ can be obtained by looking at the kernel of the canonical Liouville form on $\mathbb{C}^2$ restricted to $S^3$.
We will consider the overtwisted contact structure $\xi$ on $S^3$ with Hopf invariant $-1$, that is not isotopic to the standard tight contact structure as an oriented plane field, but can be obtained from it by modifying the standard tight contact form by performing a Lutz twist around a fibre of the Hopf fibration $S^3\to S^2$, \cite[Section 4.3]{Geiges}.
\end{exmp}

\subsection{Non-loose unknots}\label{sec_loose}

\textit{Legendrians} in contact $3$-manifolds are given by knots that are tangent to the contact distribution.
In the overtwisted case a Legendrian knot $K$ is called \textit{loose} if there exists an overtwisted disk in the complement of $K$ and \textit{non-loose} or exceptional if it intersects every overtwisted disk.

Non-loose Legendrian unknots in $(S^3,\xi)$, i.e., non-loose Legendrian knots smoothly isotopic to a Hopf fibre, are classified up to isotopy and orientation by their classical invariants, the \textit{Thurston-Bennequin invariant} $\mathrm{tb}(K)\in\Z$ and the \textit{rotation number} $\mathrm{rot}(K)\in\Z$, see e.g. \cite[Section 3.5]{Geiges} for the precise definitions.

In particular, it was proved by Eliashberg and Fraser \cite{Eliashberg09}, that $\xi$ is the only contact structure on $S^3$ that admits non-loose unknots.
Moreover, they showed that up to contactomorphism non-loose Legendrian unknots are classified by their classical invariants, which take the values $\mathrm{tb}(K)=n$ and $\mathrm{rot}(K)=\pm(n-1)$ for any integer $n\geq 1$.
This result was improved to a classification up to Legendrian isotopy in \cite{Etnyre13, Vogel18}.

Examples of non-loose unknots can be constructed as follows \cite{Dymara01, Eliashberg09}.
Note that the complement of a link of two Hopf fibres in $S^3$ is diffeomorphic to $(0,1)\times T^2$.
Then in coordinates outside of two Hopf fibres $\xi$ is the kernel of the $1$-form
$$\alpha=\sin\left(\frac{3\pi}{2}r\right)dx_1+\cos\left(\frac{3\pi}{2}r\right)dx_2.$$
As pointed out in \cite{Eliashberg09} the torus $\{r=r_0\}$ is foliated by loose Legendrian unknots with Thurston-Bennequin invariant $n$ and rotation number $\pm(n-1)$ if and only if $\tan\left(\frac{3\pi r_0}{2}\right)\in \left\lbrace -n,\frac{-1}{n}\right\rbrace$.

The contact structure $\xi$ is known to be invariant under the $S^1$-action given by the Hopf fibration \cite{Lutz77}.
In these coordinates the Hopf fibration is induced by the vector field $X:=\partial_{x_1}+\partial_{x_2}$.
The function 
$$h(r,x_1,x_2):=\alpha(X)=\sin\left(\frac{3\pi}{2}r\right)+\cos\left(\frac{3\pi}{2}r\right)$$
extends to a smooth $S^1$-invariant function on $S^3$ that was used by Lutz to classify $S^1$-invariant contact forms.
Then the above tori foliated by non-loose unknots with Thurston-Bennequin invariant $n$ are given by 
$$h^{-1}\left(\frac{1-n}{\sqrt{1+n^2}}\right), \ h^{-1}\left(\frac{1-\frac{1}{n}}{\sqrt{1+\frac{1}{n^2}}}\right).$$

Further note that in the complement of the Hopf link, generic Legendrian knots are determined by their front projection $(0,1)\times T^2\rightarrow T^2$, which is a closed curve with transverse double points and cusps.
Any such curve can be lifted to a Legendrian knot in $(S^3,\xi)$.
The front projections of isotopic Legendrians knots are related by isotopies, Reidemeister moves and a change of the front projection that occurs when an isotopy moves through a component of the Hopf link, see \cite[Section 2.5]{Vogel18}.

\subsection{Orderability and open book decompositions}

In \cite{Eliashberg00} Eliashberg and Polterovich introduced the notion of orderability for closed contact manifolds.
Let $(M,\zeta)$ be a closed cooriented contact manifold.
A \textit{contactomorphism} is a diffeomorphism of $M$ preserving $\zeta$.
Consider the universal cover of the group of contactomorphisms $\widetilde{\mathcal{G}}:=\widetilde{\mathrm{Cont}}_0(M,\zeta)$, which can be identified with isotopy classes of paths of contactomorphisms starting at the identity.
This allows to define a conjugation invariant relation on $\widetilde{\mathcal{G}}$ by declaring $\mathrm{id}\preccurlyeq \tilde{f}$ if and only if there exists a path $(f_t)_{t\in [0,1]}$ that is non-negatively transverse to the contact distribution in the sense that for any given positive contact form $\alpha$ we have
$$\alpha\left(\frac{d}{dt}f_t\right)\geq 0.$$
Due to the invariance under conjugation this allows to define a bi-invariant relation $\tilde{f}\preccurlyeq\tilde{g}$ on  $\widetilde{\mathcal{G}}$.
A contact manifold is called \textit{orderable} if this relation is a partial order on $\widetilde{\mathcal{G}}$.
This is equivalent to the existence of a contractible loop of contactomorphisms that is positively transverse to the contact distribution, \cite{Eliashberg00}.

By now orderability is known in many cases, such as for example certain prequantization spaces \cite{Eliashberg00}, hypertight contact manifolds \cite{Albers15} or manifolds with a filling that has non-vanishing symplectic cohomology \cite{colin19}.

Less is known about non-orderable contact manifolds.
So far, all examples contained in the literature are fillable by sub-critical Liouville manifolds \cite{Eliashberg06, Hedicke24}.
On the other hand overtwisted $3$-manifolds (which are never fillable) are the first example of contact manifolds suspected to be non-orderable in \cite[Section 1.9]{Eliashberg00}.
Several partial results indicate that this might be indeed the case.
In \cite{Borman15} and \cite{Liu20} the authors showed that overtwisted contact manifolds are non-orderable with respect to a weaker relation than the one introduced above.
The first examples of possibly non-contractible positive loops of contactomorphisms where constructed by Casals and Presas in \cite{Casals16}.

In this paper we use an approach to non-orderability developed in \cite{Hedicke24} that uses open book decompositions.
An \textit{open book decomposition} is a way to decompose a manifold $M$ into a co-dimension two submanifold $B$ called the binding and a locally trivial fibration $\theta\colon M\setminus B\rightarrow S^1$.
In \cite{Giroux00} Giroux introduced the notion of open book decompositons supporting a contact structure.
Roughly this means that there exists a contact form whose differential restricts to a symplectic form on the pages $\theta^{-1}(p)$ and turns them into Liouville domains, see \cite{Etnyre04}, \cite[Section 4.4.2]{Geiges} for further details.
As in \cite{Hedicke24} we will consider skeletons of the pages with respect to some ideal Giroux form \cite{Courte18} turning them into Liouville manifolds.
To prove the non-orderability of $(S^3,\xi)$ we will use the following result proved in \cite[Theorem 1.10]{Hedicke24}.

\begin{theorem}
Let $(M,\zeta)$ be a closed cooriented contact manifold supported by an open book decompositon $(B,\theta)$.
Fix a page of the open book decomposition with skeleton $\Lambda_0$.
Suppose there exist a contact form $\alpha$ and a contactomorphism $f$ that displaces $\Lambda_0$ from its image under the Reeb flow of $\alpha$.
Then $(M,\zeta)$ is non-orderable.
\end{theorem}

\section{Proofs}

\subsection{An open book decomposition for the sphere}

It is well-known that $(S^3,\xi)$ can be constructed using the following abstract open book decomposition.

Let $A=[-2,2]\times S^1$ be an annulus with coordinates $(s,\phi)\in [-2,2]\times \R/(2\pi\Z)$ and standard orientation.
Consider a smooth map $F\colon A\rightarrow A$ such that $F(s,\phi)=(s,\phi+f(s))$, where $f\colon [-2,2]\rightarrow \R$ is a smooth map such that $f(s)=-2\pi s$ on $[0,1]$, $f\equiv 0$ near $s=-2$, $f\equiv -2\pi$ near $s=2$ chosen such that $F$ is isotopic to a negative Dehn twist.

Let
$$\mathrm{MT}(A,F):=[0,1]\times A/_{(0,F(p))\sim(1,p)}$$
be the mapping torus of the map $F$.
Note that $\partial \mathrm{MT}(A,F)\cong \partial A\times S^1$.
Then 
$$S^3\cong \mathrm{MT}(A,F)\cup_\psi\left(D^2\times S^1\sqcup D^2\times S^1\right),$$
where $\psi\colon \partial\mathrm{MT}(A,F)\rightarrow\partial\left(D^2\times S^1\sqcup D^2\times S^1\right)$ is a diffeomorphism, used to identify the boundaries, see \cite{Etnyre04, Geiges}.

A contact form adapted to this open book decomposition can be obtained as follows.
Let $\lambda=sd\phi$ be the standard Liouville form on $A$, viewing $A$ as a co-disk bundle of $S^1$.
Let $\mu \colon [0,1]\rightarrow [0,1]$ be a smooth function such that $\mu=1$ near $0$ and $\mu=0$ near $1$.
Then
$$\mu(\theta)\lambda+(1-\mu(\theta))F^{\ast}\lambda$$
defines a $1$-form on $[0,1]\times A$ that descends to a smooth $1$-form on $\mathrm{MT}(A,F)$.
Here $\theta\in [0,1]$ denotes the coordinate descending to the fibration \\
$\theta\colon \mathrm{MT}(A,F)\rightarrow S^1$ in the quotient.

A contact form on $\mathrm{MT}(A,F)$ is then given by
$$\alpha:=Cd\theta + \mu(\theta)\lambda+(1-\mu(\theta))F^{\ast}\lambda$$
for a sufficiently large constant $C>0$.
This contact form extends to a contact form on $S^3$ supporting $\xi$, see \cite{Etnyre04, Geiges}.

An ideal Giroux form $\beta$ as used in \cite[Theorem 1.10]{Hedicke24} can be obtained by replacing $\lambda=sd\phi$ with $\tilde{\lambda}=\tan\left(\frac{\pi s}{4}\right)d\phi$ in the above expression for $\alpha$.

\begin{lemma}\label{lem_torus}
The skeleton of the page $\theta^{-1}(0)$ with respect to $\beta$ is a non-loose Legendrian unknot $\Lambda_0$ with $\mathrm{tb}(\Lambda_0)=1$ and $\mathrm{rot}(\Lambda_0)=0$.
Its image under the Reeb flow of $\alpha$ is the pre-Lagrangian torus $T$ that is the zero-level set of the function $p\mapsto\alpha_p(X_p)$.
Here $X$ denotes the vector field inducing the Hopf fibration.
\end{lemma}

\begin{proof}
For the page $W:=\theta^{-1}(0)$ the skeleton $\Lambda_0$ with respect to $\beta|_W$ is given by the circle $\{s=0,\theta=0\}\subset W\cong A$, since $\beta|_W=\tilde{\lambda}=\tan\left(\frac{\pi s}{4}\right)d\phi$.
Note that for $s\in [0,1]$ the contact form is given by
$$\alpha=C d\theta + sd\phi - 2\pi\mu(\theta)sds.$$
Hence the Reeb vector field in this region is point-wise collinear to
$$\partial_{\theta}-2\pi\mu'(\theta) s\partial_{\phi}.$$
In particular, the Reeb flow is periodic at $\Lambda_0$ and 
$$T:=\bigcup\limits_{t\in\R}\phi_t^R(\Lambda_0)=\{s=0\}$$
is a pre-Lagrangian torus in $(S^3,\xi)$.

As pointed out in \cite[Remark 1.2]{Hedicke24}, $\Lambda_0$ is non-loose.
This can be seen as follows.
Due to \cite[Proposition 3.5]{Courte18}, \cite[Lemma 2.7]{Hedicke24} there exists a contact flow that contracts every compact subset in the complement of $\Lambda_0$ to an arbitrarily small neighbourhood of another skeleton $\Lambda_1$.
The skeleton $\Lambda_1$ is a Legendrian knot and hence admits a standard neighbourhood contactomorphic to the $1$-jet bundle of the circle.
This neighbourhood does not contain any overtwisted discs and therefore there can not be any overtwisted discs contained in the complement of $\Lambda_0$.

Further $\Lambda_0$ is smoothly isotopic to a binding component of the open book which is a Hopf fibre.
Hence $\Lambda_0$ is a non-loose unknot.

Clearly, as $\xi|_{\Lambda_0}$ is tangent to $W$, $\Lambda_0$ (considered as an oriented knot with the orientation given by $\partial_{\phi}$) has vanishing rotation number and hence Thurston-Bennequin number $1$.

Note that the angular coordinate $\phi$ in the pages induces a contact $S^1$-action on $\mathrm{MT}(A,F)$ that extends smoothly to a contact $S^1$-action on $(S^3,\xi)$ which is given by the Hopf fibration $\pi\colon S^3\rightarrow S^2$.
As described in Section \ref{sec_loose} the $S^1$ invariant smooth function $h\colon S^3\rightarrow \R$ given by $h=\alpha(\partial_{\phi})$ was used by Lutz \cite{Lutz77} to classify $S^1$-invariant contact forms.

The torus $T$ is given by $h^{-1}(0)$ and coincides with the torus foliated by non-loose unknots with Thurston-Bennequin number $1$ considered in \cite{Dymara01, Eliashberg09, Vogel18}.
\end{proof}

\subsection{Displacing the skeleton}

Non-loose unknots have been classified in \cite{Dymara01, Eliashberg09, Etnyre13, Vogel18}.
In fact $\xi$ is the unique overtwisted contact structure on $S^3$ admitting non-loose Legendrian unknots.
The classification results imply that in order to displace $\Lambda_0$ from $T$ it is sufficient to find a non-loose Legendrian unknot with rotation number $0$ that is disjoint from $T$.

The following Lemma can be obtained by considering the isotopy described in \cite[Section2.5]{Vogel18}

\begin{lemma}\label{lem_displace}
There exists a contact isotopy $(f_t)_{t\in[0,1]}$ and the skeleton of a page $\Lambda_1$ different from $\Lambda_0$ such that $f_1(\Lambda_0)\cap T=\emptyset$ and $f_t(\Lambda_0)\cap \Lambda_1=\emptyset$ for all $t\in[0,1]$.
\end{lemma} 

\begin{proof}

As discussed in Section \ref{sec_loose}, there exist two level sets of $h$ foliated by non-loose Legendrian unknots with Thurston-Bennequin invariant $2$ and rotation number $1$.
Pick such a Legendrian unknot $K'$.
Note that $K'$ does not intersect the torus $T$.

Let $K$ be a negative stabilization of $K'$.
The stabilization can be performed in a small Darboux neighbourhood around $K'$.
We can therefore assume that $K$ is disjoint from $T$.

Then by \cite[Theorem 1.12]{Etnyre13} or \cite[Section 2.5]{Vogel18}, $K$ is a non-loose unknot with $\mathrm{tb}(K)=1$ and $\mathrm{rot}(K)=0$ and in particular Legendrian isotopic to $\Lambda_0$ (up to orientation).

In terms of the front projection $(0,1)\rightarrow T^2$ in the complement of two Hopf links the isotopy can be described as in \cite[Figure 6]{Vogel18}.
As pictured in Figure 1 below, there exists a non-loose Legendrian unknot $\Lambda_1$ in $T$ that is disjoint from the isotopy and the skeleton of a page (red curve).

Here the first arrow indicates the negative stabilization of $K'$ which is moving closer to a Hopf fibre in the third picture.
The fourth picture shows the front projection of the isotopy after crossing the Hopf fibre.
After applying two Reidemeister moves in picture 5, we end up with $\Lambda_0$ whose front projection is a $(1,1)$ curve in $T^2$.
Note that the isotopy only intersects the pre-Lagrangian torus $T$ at the points where the front projection has slope $1$.
This never happens at the intersection points with the front projection of $\Lambda_1$.

By the isotopy extension theorem the Legendrian isotopy extends to a contact isotopy $(f_t)_{t\in[0,1]}$ that can be assumed to be compactly supported in the complement of $\Lambda_1$.
\end{proof}

\begin{center}
\begin{figure}\label{Fig_isotopy}
\includegraphics[scale=0.4]{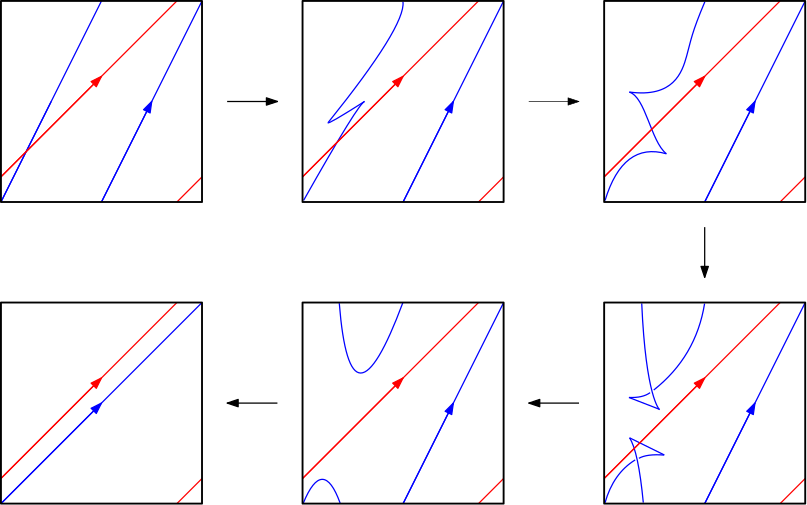}
\caption{The isotopy from the negative stabilization of $K'$ to $\Lambda_0$. The red curve is the front projection of a skeleton on $T$ disjoint from the isotopy.}
\end{figure}
\end{center}

\subsection{Proof of Theorem \ref{thmmain} and Theorem \ref{thm_translated}}

In Lemma \ref{lem_torus} we showed that there exist a supporting open book and a contact form for $(S^3,\xi)$ such that the skeleton of a page is a non-loose Legendrian unknot and its image under the Reeb flow is the standard pre-Lagrangian torus foliated by non-loose Legendrian unknots with Thurston-Bennequin invariant one.
Lemma \ref{lem_displace} shows that this skeleton can thus be displaced from its image under the Reeb flow.
The non-orderability of $(S^3,\xi)$ follows from \cite[Theorem 1.10]{Hedicke24}.
This proves Theorem \ref{thmmain}.

Lemma \ref{lem_displace} further shows that the skeleton can be displaced from its image under the Reeb flow in the complement of a second skeleton.
Then the proof of Theorem \ref{thm_translated} works completely analogous to the one of \cite{Cant22} and of \cite[Theorem 4.4]{Hedicke24}.
A contactomorphism without translated points can be constructed using the contact flow in \cite[Proposition 3.5]{Courte18}, \cite[Lemma 2.7]{Hedicke24} and the contactomorphism displacing the skeleton.

\bibliographystyle{siam}

\end{document}